\documentclass[10pt,a4paper]{amsart}

\usepackage{amsmath, amsthm, amssymb}
\usepackage{color,url}
\usepackage[all]{xy}
\usepackage[latin1]{inputenc}

\DeclareMathOperator{\Aut}{Aut}

\DeclareMathOperator{\GL}{GL}

\DeclareMathOperator{\SL}{SL}

\DeclareMathOperator{\Comm}{Comm}

\DeclareMathOperator{\hol}{hol}
\DeclareMathOperator{\comp}{comp}
\DeclareMathOperator{\stab}{stab}

\newcommand{\GBS}{\mathcal{GBS}}
\newcommand{\NN}{\mathbf{N}}

\newcommand{\ZZ}{\mathbf{Z}}
\newcommand{\RR}{\mathbf{R}}

\newcommand{\QQ}{\mathbf{Q}}
\newcommand{\Af}{\mathbf{A}}
\newcommand{\Bf}{\mathbf{B}}

\newtheorem{theorem}{Theorem}

\newtheorem{lemma}[theorem]{Lemma}

\theoremstyle{remark}

\newtheorem*{claim*}{Claim}

\newtheorem*{observation*}{Observation}

\newtheorem{remark}[theorem]{Remark}

\newtheorem{definition}[theorem]{Definition}

\title{The Haagerup property is not invariant under quasi-isometry}
\author[Mathieu Carette]{Mathieu Carette \\ \\ With Appendix A by Kevin Whyte; and Appendix B by Sylvain Arnt, Thibault Pillon and Alain Valette}

\address{Rockestate SRL, Rue du Commerce 31, 1000 Brussels, Belgium}
\email{\tt mathieu.carette@gmail.com}

\address{Dept. of Mathematics, Statistics, and Computer Science, University of Illinois at Chicago, USA}
\email{\tt kwhyte@uic.edu}

\address{P\^ole Sup' Sainte Croix - Saint Euverte, 28 rue de l'Etelon, F-45043 Orl\'eans , France}

	\email{\tt arnt.sylvain@gmail.com}
	
	\address{Haute \'ecole d'Ing\'enierie et de Gestion du Canton de Vaud, Route de Cheseaux 1, CH-1400 Yverdon-les-Bains, Switzerland}
	
	\email{\tt thibault.pillon@heig-vd.ch}
	
	\address{Institut de Math\'ematiques, UniMail, 11 Rue Emile Argand, CH-2000 Neuch\^atel, Switzerland}
	
	\email{\tt alain.valette@unine.ch}

\thanks{M.C.\ was a Postdoctoral Researcher of the F.R.S.-FNRS (Belgium).}
\thanks{T.P. was supported by grant FN 200020-149261/1 of the Swiss SNF}

\subjclass[2010]{20F65, 20E08, 22D05, 22D10}
\keywords{Haagerup property, a-T-menability, weak amenability, quasi-isometry, generalized Baumslag-Solitar groups, equivariant $L^p$ compression} 
\begin{document}
	
	\begin{abstract} Using the work of Cornulier-Valette and Whyte (see Appendix A for the latter), we show that neither the Haagerup property nor weak amenability is invariant under quasi-isometry of finitely generated groups. Appendix B shows, using the same examples, that the same holds for vanishing of the equivariant $L^p$-compression.
	\end{abstract}
	\maketitle
	
	Central both in geometric and in measured group theory, is the study of invariants for finitely generated groups, either under quasi-isometry (QI) or under measure equivalence (ME). Recall the relevant definitions
	\begin{definition}Let $\Gamma,\Lambda$ be finitely generated groups. 
	\begin{itemize}
	\item The groups $\Gamma$ and $\Lambda$ are {\bf quasi-isometric} if there exists a locally compact space $X$ equipped with commuting, free, proper actions of $\Gamma$ and $\Lambda$ by homeomorphisms, such that the orbit spaces $\Gamma\backslash X$ and $X/\Lambda$ are compact.
	\item The groups $\Gamma$ and $\Lambda$ are {\bf measure-equivalent} if there exists a standard measure space $(Y,\mu)$ equipped with commuting, free, measure-preserving actions of $\Gamma$ and $\Lambda$ by Borel automorphisms, having finite measure Borel fundamental domains $Y_\Gamma,Y_\Lambda$ respectively. 
		\end{itemize}\end{definition}

	In spite of the similarity of definitions, in general quasi-isometry does not imply measure equivalence, and vice versa. In fact a comparison of the basic invariants reveals strikingly little intersection (See \cite{BHV08, CCJJV01, CornulierHarpe, Furman11} and references therein for the basics on QI and ME invariants, as well as Section~\ref{sec:defs} for definitions).

	\begin{itemize}
		\item QI invariants: amenability, number of ends, growth, hyperbolicity, finite presentability and Dehn function, asymptotic cones,...
		\item ME invariants: amenability, property (T), ratio of $L^2$-Betti numbers, the Haagerup property (also called a-T-menability), weak amenability and the Cowling-Haagerup constant,... 
	\end{itemize}
	
	With the notable exception of amenability, all QI-invariant properties listed  above are known not to be ME invariant. In the other direction, it is known that neither property (T) nor the ratio of $L^2$-Betti numbers are QI invariants. It is therefore natural to ask whether the Haagerup property and weak amenability, which are both natural generalizations of amenability, are also QI invariants. We show that none of those properties are QI invariant, therefore settling a question raised in \cite[p. 774]{CornulierTesseraValette07}.
	
	\begin{theorem} \label{theorem:main} There exists two finitely generated groups $\Gamma$, $\Lambda$ which are quasi-isometric such that $\Gamma$ has the Haagerup property and is weakly amenable and $\Lambda$ has neither of those properties.
	\end{theorem}
	
	Our examples are fundamental groups of graphs of $\ZZ^2$'s. On the one hand the quasi-isometric classification of such groups was obtained by Whyte \cite{Whyte10}\footnote{Since the preprint \cite{Whyte10} did not appear so far, Kevin Whyte was kind enough to write up the relevant part as Appendix A.}. On the other hand Cornulier-Valette \cite{CornulierValette12} characterized those groups with the Haagerup property (which, for this class of groups, turns out to be equivalent to weak amenability). It turns out that both answers depend on a holonomy map discussed in Section~\ref{sec:GBS}. We also point out the key role of locally compact groups in Section~\ref{sec:LC}. In particular allowing for unimodular compactly generated locally compact groups rather than finitely generated groups, we give examples as in Theorem~\ref{theorem:main} where both groups are of the form $\RR^2 \rtimes F$ with $F$ a finitely generated free group (see Remark~\ref{remark:LCexamples}). In Appendix B, Arnt, Pillon and Valette use our finitely generated examples to show that for $1 \leq p \leq 2$ the vanishing of equivariant $L^p$ compression is not a quasi-isometry invariant.
	
	\subsection{Definitions}\label{sec:defs} We briefly recall here definitions and facts regarding the properties of interest.
	
	\begin{definition} Let $G$ be a locally compact group, second countable group $G$.
	\begin{itemize}
	\item The group $G$ is \textbf{amenable} if there exists a sequence $(\varphi_n)_{n>0}$ of normalized positive definite functions\footnote{A continuous function $\varphi$ on $G$ is {\bf normalized positive definite} if there exists a strongly continuous unitary representation $\pi$ on some Hilbert space $\mathcal{H}$, and a unit vector $\xi\in\mathcal{H}$, such that $\varphi(g)=\langle\pi(g)\xi|\xi\rangle$ for every $g\in G$.} on $G$, with compact support, such that $\varphi_n\xrightarrow[n\to \infty]{} 1 $ uniformly on compact sets.
	\item The group $G$ has the \textbf{Haagerup property} (or is \textbf{a-T-menable}) if there exists a sequence $(\varphi_n)_{n>0}$ of normalized positive definite functions on $G$, vanishing at infinity, such that $\varphi_n\xrightarrow[n\to \infty]{} 1 $ uniformly on compact sets.
	
	\item The group $G$ is \textbf{weakly amenable} if there is a constant $\Lambda$ and a sequence $(\varphi_n)_{n>0}$  of continuous functions with compact support on G, such that 
	 $\varphi_n \xrightarrow[n\to \infty]{} 1$ uniformly on compact sets and 
	 \begin{equation} \label{eq:weak_amen}
 \|\varphi_n\|_{\text{cb}} \leq \Lambda \text{ for all } n>0.
\end{equation} 
	where $\|\varphi\|_{\text{cb}}$ denotes the completely bounded norm of a function $\varphi$ viewed as a multiplier of the Fourier algebra $A(G)$. The \textbf{Cowling-Haagerup} constant $\Lambda(G)$ is the infimum over all possible $\Lambda$'s as in~\eqref{eq:weak_amen}, with the convention $\Lambda(G)=\infty$ if no such $\Lambda$ exists.
	\end{itemize}
	\end{definition} 
	
	Clearly amenable groups have the Haagerup property; they are also weakly amenable with Cowling-Haagerup constant 1, as $\|\varphi\|_{\text{cb}}=1$ when $\varphi$ is a normalized positive definite function on $G$. However, both of these properties also encompass many non-amenable groups such as free groups or full isometry groups of locally finite trees. Neither does the Haagerup property imply weak amenability (see \cite{CorStaVal}), nor vice versa. Whether $\Lambda(G)=1$ implies the Haagerup property for $G$ remains an open question.

	
	\subsection{Generalized Baumslag-Solitar groups} \label{sec:GBS} Fix an integer $n \in \NN$. We consider the class $\GBS_n$ of groups $\Gamma$ acting cocompactly on a locally finite tree $T$ such that all (vertex and edge) stabilizers are isomorphic to $\ZZ^n$. Equivalently, a group $\Gamma$ is in $\GBS_n$ if it is the fundamental group of a finite graph of groups where all edge and vertex groups are isomorphic to $\ZZ^n$ \cite{Serre80}. 
	
	Let $\Gamma$ and $T$ be as above. Fix a vertex $v \in T$ and an isomorphism $\Gamma_v \cong \ZZ^n$. As all inclusions of edge groups into vertex groups have finite index, $\Gamma$ commensurates $\Gamma_v$. So the action of $\Gamma$ on $\Gamma_v$ by commensuration induces a homomorphism to the abstract commensurator $\hol : \Gamma \to \Comm(\Gamma_v) = \GL_n(\QQ)$, which we call the \textbf{holonomy map}\footnote{This is also sometimes called the modular homomorphism, especially for the class $\GBS_1$ \cite{Levitt07}.}. Note that $h$ is well-defined up to conjugation in $\GL_n(\QQ)$ (with the different possibilities coming from a different choice of basepoint $w$ and a different identification $\Gamma_w \cong \ZZ^n$). In fact, the commensurability class of $\Gamma_v$ inside $\Gamma$ does not depend on the chosen tree $T$ (as soon as $\Gamma$ is not amenable, equivalently $\Gamma$ stabilizes no vertex, no line and no end of $T$) see e.g.\ the proof of Lemma~8.5 in \cite{GuirardelLevitt07}. From now on, we view the holonomy as a map from $\Gamma$ to $\GL_n(\RR)$.
	
	We say that $\Gamma, \Gamma'\in \GBS_n$ have \textbf{Hausdorff equivalent holonomy} if there is a compact subset $K \subset \GL_n(\RR)$ and some $g \in \GL_n(\RR)$ such that $\hol(\Gamma) \subset g \hol'(\Gamma')g^{-1} K$ and $g \hol'(\Gamma') g^{-1} \subset \hol(\Gamma) K$. In other words $\hol(\Gamma)$ is at finite Hausdorff distance from some conjugate $\hol'(\Gamma')$ for the word metric corresponding to some (equivalently any) compact generating set of $\GL_n(\RR)$.
	
	Whyte showed that the QI classification of $\GBS_n$ groups is essentially governed by the holonomy map.  	
	\begin{theorem}[{\cite[Theorem 0.1]{Whyte10}}] \label{theorem:QIclassif} Among the class of groups in $\GBS_n$ whose Bass-Serre tree $T$ has infinitely many ends the following holds:
		\begin{enumerate} 
			\item If two groups are quasi-isometric then they have Hausdorff equivalent holonomy.	\item (see Theorems \ref{FoldedQI} and \ref{trichotomy} in Appendix A) The set of groups within a given Hausdorff equivalence class of holonomy divides into three quasi-isometry invariant subclasses:
			\begin{enumerate}
				\item \label{case:proper} Those which are of the form $\ZZ^n \rtimes F$ for $F$ a free subgroup of $\GL_n(\ZZ)$.
				\item \label{case:amenable} Those which are virtually ascending HNN-extensions of some endomorphism $E : \ZZ^n \to \ZZ^n$. These are classified up to QI in \cite{FarbMosher00}.
				\item \label{case:folded} All groups not of the first two forms, all of which are in a single quasi-isometry class.
			\end{enumerate}
		\end{enumerate}
	\end{theorem}
	\begin{remark} \label{remark:amenable_proper} Groups of the form \eqref{case:amenable} are exactly those groups which are amenable. Groups of the form \eqref{case:proper} have a holonomy with discrete image in $\GL_n(\RR)$.
	\end{remark}
	
	Cornulier and Valette showed that both the Haagerup property and weak amenability of a $\GBS_n$ group is determined by the image of its holonomy. 
	\begin{theorem}[{\cite[Theorem 1.6]{CornulierValette12}}] \label{theorem:CVclassif} Let $\Gamma \in \GBS_n$, with holonomy $\hol : \Gamma \to \GL_n(\RR)$. Then the following are equivalent:
		\begin{enumerate}
			\item \label{item:Haagerup} $\Gamma$ has the Haagerup property.
			\item $\Gamma$ is weakly amenable.
			\item $\Gamma$ has Cowling-Haagerup constant $1$.
			\item \label{item:hol_amenable} $\overline{\hol(\Gamma)}$ is amenable.
		\end{enumerate}
	\end{theorem}
	
	\begin{proof}[Proof of Theorem~\ref{theorem:main}] 
	Let $X = \ZZ^2 = \langle a,b\mid [a,b]\rangle$ and consider the subgroup $Y  = \langle a,b^2 \rangle < X$. Consider furthermore the following matrices in $\SL_2(\QQ)$:
	\[ H = \left( \begin{array}[pos]{cc} 2 & 0 \\ 0 & \frac{1}{2} \end{array} \right) ; 
		P = \left( \begin{array}[pos]{cc} 1 & 1 \\ 0 & 1 \end{array} \right) \text{ and }
		E = \left( \begin{array}[pos]{cc} 0 & 1 \\ -1 & 0 \end{array} \right) .\]
	Finally, consider the following graphs of groups
	\[  \Af = {\begin{xy}\xymatrix{ 
			 X \ar@{->}@(r,dr)^{Y}^<<<{id}^>>>{H} \ar@{->}@(dl,l)^{X}^<<<{id}^>>>{P}
			 }\end{xy}} 
		\text{ and } 
		\Bf = {\begin{xy}\xymatrix{ 
			 X \ar@{->}@(r,dr)^{Y}^<<<{id}^>>>{H} \ar@{->}@(dl,l)^{X}^<<<{id}^>>>{P} \ar@{->}@(ul,ur)^{X}^<<<{id}^>>>{E}
			 }\end{xy}}
	\]
	and let $\Gamma = \pi_1(\Af)$ and $\Lambda = \pi_1(\Bf)$. In other words, the groups $\Gamma$ and $\Lambda$ are defined by the following presentations:
	\begin{align*} \Gamma = \langle a,b,h,p &\mid ab=ba, a^h = a^2, (b^2)^h = b, a^p = a, b^p = ab \rangle \\
	\Lambda = \langle a,b,h,p,e &\mid ab=ba, a^h = a^2, (b^2)^h = b, a^p = a, b^p = ab, a^e=b^{-1}, b^e = a \rangle
	\end{align*}	
	(where $x^g$ stands for $gxg^{-1}$).The holonomy maps of $\Gamma$ and $\Lambda$ send $a,b$ to $1$, and $h,p,e$ to the matrices $H,P,E\in \SL_2(\RR)$ respectively. It follows from  $H^kPH^{-k} = \left( \begin{array}[pos]{cc} 1 & 4^k \\ 0 & 1 \end{array} \right)$ that
	\begin{equation} \overline{\hol(\Gamma)} = \overline{\langle H, P\rangle} =  \left.
	\left\{ \left( \begin{array}[pos]{cc} 2^k & x \\ 0 & 2^{-k} \end{array} \right) \right| \begin{array}[pos]{c} k \in \ZZ \\ x \in \RR \end{array} \right\} \label{equation:T}
	\end{equation}
	is solvable and in particular amenable. On the other hand $\langle P,E \rangle = \SL_2(\ZZ)$ contains a free discrete subgroup of $\SL_2(\RR)$ so that $\overline{\hol(\Lambda)} = \overline{\langle H, P, E\rangle} \subset \SL_2(\RR)$ is not amenable. In fact $\overline{\langle H, P, E\rangle} = \SL_2(\RR)$. 
	In view of Theorem~\ref{theorem:CVclassif} the group $\Gamma$ has the Haagerup property and is weakly amenable, while $\Lambda$ has neither of these properties.
	
	We now check using Theorem~\ref{theorem:QIclassif} that $\Gamma$ and $\Lambda$ are quasi-isometric. First observe that the Bass-Serre trees of the graphs of groups $\Af$ and $\Bf$ are the $6$ and $8$-regular trees respectively, so that they have infinitely many ends. Next observe that there is a compact subset $K \subset \SL_2(\RR)$ such that $\hol(\Gamma)K = \SL_2(\RR)$, so in particular $\Gamma$ and $\Lambda$ have Hausdorff equivalent holonomies. Finally, the subgroup $\langle h,p\rangle$ is a free subgroup of $\Gamma$ and $\Lambda$, so that neither $\Gamma$ nor $\Lambda$ is amenable; moreover the image of $\langle h,p \rangle$ by the respective holonomies is non-discrete in $\GL_2(\RR)$. It follows from Remark~\ref{remark:amenable_proper} that both $\Gamma$ and $\Lambda$ fall into the class \eqref{case:folded} of Theorem~\ref{theorem:QIclassif}, hence they are quasi-isometric to each other.
	\end{proof}
	
	\begin{remark} While we show that the Cowling-Haagerup constant is not a QI invariant, it follows from Theorem~\ref{theorem:CVclassif} that the class $\GBS_n$ does not contain examples of two quasi-isometric groups with different \emph{finite} Cowling-Haagerup constant.
	\end{remark}
	
	\subsection{The role of locally compact groups} \label{sec:LC}
	For the convenience of the reader, we briefly outline ideas behind Cornulier-Valette's equivalence \eqref{item:Haagerup} $\Leftrightarrow$ \eqref{item:hol_amenable} in Theorem~\ref{theorem:CVclassif}. 
	Doing so we emphasize the role played by non-discrete locally compact groups. 
	
	\begin{remark} The QI relation naturally extends to compactly generated locally compact  groups, while ME makes sense for unimodular second countable locally compact groups. All invariants mentioned in the introduction extend naturally to these larger settings, with the caveat that amenability is a QI invariant only within \emph{unimodular} compactly generated locally compact groups (see \cite{KyedPetersenVaes13} for the ME invariance of the ratio of $L^2$-Betti numbers).
	\end{remark}
	
	Recall the definitions of the groups $\Gamma, \Lambda$ and the matrices $H,P,E \in \SL_2(\RR)$ from the previous section. Let $\varphi : F_2 \to \SL_2(\RR)$ and $\psi : F_3 \to \SL_2(\RR)$ be homomorphisms mapping a basis of $F_2$ to $(H,P)$, respectively a basis of $F_3$ to $(H,P,E)$. Finally, let $G = \RR^2 \rtimes_\varphi F_2$ and $L = \RR^2 \rtimes_\psi F_3$.
	
	To show that $\Gamma$ has the Haagerup property Cornulier and Valette realize $\Gamma$ as a discrete subgroup in the locally compact group $\Aut(T) \times G$ where $T$ is the ($6$-regular) Bass-Serre tree of the graph of groups $\Af$; it is known that $\Aut(T)$ has the Haagerup property. They proceed to show that $G$ has the Haagerup property, so that $\Gamma$ also has the Haagerup property. On the other hand, the group $\Lambda$ contains a finite index subgroup of $\ZZ^2 \rtimes \SL_2(\ZZ)$ which has property (T) relative to the subgroup $\ZZ^2$, so that $\Lambda$ cannot have the Haagerup property.
	
	
	\begin{remark} \label{remark:LCexamples} The semidirect products $G$ and $L$ are arguably simpler (non-discrete) examples showing that the Haagerup property and weak amenability are not QI invariants among \emph{unimodular} second countable, compactly generated, locally compact groups. Such examples cannot be obtained by cocompact inclusions $H < H'$ as this would force $H$ and $H'$ to be ME. Let $T_2 = \overline{\langle H,P \rangle}$ be the group described in equation \eqref{equation:T}: although it provides an example of two quasi-isometric groups with only one of them being Haagerup, the related cocompact inclusion $\RR^2 \rtimes T_2 < \RR^2 \rtimes \SL_2(\RR)$ is uninteresting from the point of view of ME as $\RR^2 \rtimes T_2$ is not unimodular.
	\end{remark}
		
	\section*{Acknowledgements}
	This note grew out of a collaboration with Romain Tessera, whom I thank for helpful discussions. I heartily thank Kevin Whyte for agreeing to write up Appendix A. I am also grateful to Alain Valette for his comments on previous versions of this note. Finally, I thank the referees for their comments which improved the exposition.
	
\appendix
	\section{QI-classification of graphs of $\ZZ^n$'s}
	\begin{center} \textsc{by K. Whyte} \end{center}

\subsection{Introduction}
The goal of this appendix is to show that the groups constructed by Carette are quasi-isometric. This is essentially one direction of the results of section 3 of  \cite{Whyte10}, with the notation somewhat simplified as we are concerned only with graphs of $\ZZ^n$s rather than more general vertex groups. Unless otherwise specified, the groups discussed here are all {\bf graphs of $\ZZ^n$s}, meaning that, for some fixed $n$, they are fundamental groups of finite graphs of groups in which every edge and vertex group is $\ZZ^n$ and all the edge-to-vertex inclusions are finite index. The resulting Bass-Serre tree $T$ is finite valence and  for simplicity we will further only discuss the case where  $T$ has infinitely many ends (called {\bf bushy}), as the cases with $T$ finite or $2$-ended are not relevant here. The key result then becomes:
  
\begin{theorem} \label{FoldedQI} If two bushy graphs of $\ZZ^n$s, $\Gamma$ and $\Gamma'$,  both have folded holonomy maps, and after conjugating, have holonomy images at bounded Hausdorff distance then they are quasi-isometric.  
\end{theorem} 
      
   Definitions and precise statements are below.  As one can see from this statement, the holonomy map is the key ingredient.    The outline of the proof  is as follows: first we show that we can encode a quasi-isometry between the groups as quasi-isometries between the trees that respect the holonomy up to bounded distance (Lemma \ref{QIfromTreeMap}).   Second we show that the holonomy maps have three basic types, with folded holonomy the "generic" type (Theorem \ref{trichotomy}).   Third we show that folded holonomy means that one can lift maps to the holonomy image to the tree and can do so in many ways  (the somewhat technical Lemmas \ref{directedlift} and \ref{lift}).  Finally we use this lifting to show that one can recover the QI type of a folded holonomy graph of $\ZZ^n$s just from the Hausdorff class of the holonomy image (Theorem \ref{TreeMapsExist}). Together this will prove Theorem \ref{FoldedQI}.
      
 \subsection{Holonomy and model spaces}   
  Let $\Gamma$ be a graph of $\ZZ^n$s.  Observe that we can factor the holonomy map $\Gamma \to \GL_n(\RR)$ through the Bass-Serre tree $T$. Fix a base vertex $v_0 \in T$ and definite the holonomy map as the action by conjugation on $\stab(v_0) \otimes \QQ = \QQ^n$.   The holonomy $\hol(\gamma)$ is then determined by its action on any finite index subgroup of $\stab(v_0)$.   If $\gamma$ fixes a vertex $v \in T$, $\hol(\gamma)$ is trivial as we know that $\stab(v)$ is commensurable to $\stab(v_0)$ and is Abelian (in fact $\ZZ^n$) and so $\gamma$ commutes with a finite index subgroup of $\stab(v_0)$.   If we pick a fundamental domain $T_0$ in $T$ for the $\Gamma$ action  then we can factor $\hol$ as a map $h: T \to \GL_n(\RR)$ by defining $h(v) = hol(\gamma)$ for $\gamma$ such that $v \in \gamma T_0$.   We then have $h(\sigma v) = hol(\sigma) h(v)$ and $h(u) = Id$ for $u\in T_0$ (in particular $h(v_0) = Id$).
  
We connect the holonomy to the geometry by a "tree of spaces" model space for  $\Gamma$.    This is a model space $X$ for $\Gamma$ built with a {\bf vertex space} $X_v = \RR^n$ for each vertex $v \in T$, metrized in the standard way and with the vertex stabilizer acting via the standard action of $\ZZ^n$ on $\RR^n$.   Likewise, an {\bf edge space} $X_e = \RR^n \times [0,1]$ for each edge $e$ of $T$ (with the edge stabilizer acting by the standard $\ZZ^n$ action), with the ends $\RR^n \times 0$ and $\RR^n \times 1$ identified with the vertex spaces for the endpoints of of $e$ via affine maps that are equivariant with respect to the inclusions of the edge stabilizer into the corresponding vertex stabilizers.   Topologically $X$ is $T\times \RR^n$ with the diagonal action ($\Gamma$ acting on $\RR^n$ via $hol$) but the metrics are different.   

We construct a $\Gamma$-invariant metric on $X$ by using the map $h$.  Metrize $T_0 \times \RR^n$  with the product metric.  Each vertex space $X_v$ is isometric to $\RR^n$ by moving it under the $\Gamma$ action to a vertex space in $T_0 \times \RR^n$.    Two elements that move $v$ to its orbit representative in $T_0$ differ by a vertex stabilizer, and these act isometrically on $\RR^n$, so the resulting vertex space metrics are well defined.  To extend the metric to all of $X$ we need to describe how the gluing maps for $[0,1] \times \RR^n$, for each edge, behave relative to the metrics we have put on the vertex spaces.   This is equivalent to giving the closest point projection maps from one vertex space to another.  For vertex spaces $X_u$ and $X_v$ we can define this projection,  $\pi_{uv}:X_u \to X_v$, as $h(v)^{-1}h(u)$.    There are other choices in describing such a model space that produce the same metric: changing the identification of each $X_v$ with $\RR^n$ by an isometry, or globally choosing a different metric on $\RR^n$.   This is the same as changing $h$ by conjugation and composition of each $h(v)$ with an isometry of $\RR^n$.     Lemma \ref{QIfromTreeMap}  gives similar criteria for two model spaces (for possibly different groups) to be quasi-isometric. 

\begin{definition}\label{QIoverZ} Let $X$, $Y$, and $Z$ be spaces and $\alpha : X \to Z$, $\beta: Y \to Z$ maps.  We say $(X,\alpha)$ and $(Y,\beta)$ are {\bf quasi-isometric over $Z$} if there is a quasi-isometry $f:X \to Y$ such that $\beta \circ f$ and $\alpha$ are at bounded distance. 
\end{definition}

\begin{definition} Two graphs of $\ZZ^n$s, $\Gamma$ and $\Gamma'$, have coarsely equivalent holonomy if there is $A \in \GL_n(\RR)$ so that $(T,A h A^{-1})$ and $(T',h')$ are quasi-isometric over $\GL_n(\RR)$.
\end{definition}

\begin{lemma} \label{QIfromTreeMap} Two graphs of $\ZZ^n$s, $\Gamma$ and $\Gamma'$,  with coarsely equivalent holonomy are quasi-isometric. 
\end{lemma}

\begin{proof} Let $f:T\to T'$ be a quasi-isometry as in Definition \ref{QIoverZ}.

First note that as changing basepoints in the trees changes the holonomy to a conjugate, we may assume we have chosen basepoints $v_0$ and $v'_0$ so that $f(v_0) = v'_0$ (with a possibly different $A$).    \\

We build the map $F : X \to X'$ as a map of spaces.   For each $v\in T$ we will send $X_v \to X'_{f(v)}$ by the linear map $F_v = h'(f(v))^{-1}  Ah(v)$.    By assumption these linear maps are all within $R$  distance of the identity in $\GL_n(\RR)$ and hence are uniformly Lipschitz as maps $\RR^n \to \RR^n$.     

Similarly, if $e$ is an edge of $T$ with endpoints $u$ and $v$ then the closest point projection $\pi_{uv} : X_u \to X_v$ is $h(v)^{-1}h(u)$.   Likewise, the projection $\pi'$ from $X'_{f(u)}$ to $X'_{f(v)}$ is $h'(f(v))^{-1} h'(f(u))$.      From the definition of the maps on the vertex spaces, this is the same as $F_v (Ah(v))^{-1} (Ah(u)) {F_u}^{-1} = F_v \pi {F_u}^{-1}$.  Thus $F_v \pi = \pi' F_u$, meaning that these maps respect closest point projections so the combined map $F$ is Lipschitz.   

The same construction applied to the coarse inverse $g : T' \to T$ gives $G: X' \to X$.     The compositions are self-maps of $X$ and $X'$ that respect the trees of spaces decomposition and are at bounded distance from the identity on $T$ and $T'$.     For any $v \in T$ that gets maps to a nearby $v'$, the map along the $\RR^n$ fiber $X_v \to X_{v'}$ is $h(v')^{-1} h(v)$, and since this agrees with the closest point projection we have that the composition is at bounded distance from the identity.   Thus $F$ and $G$ are coarse inverses.   
\end{proof}

The converse is also true, so two bushy graphs of $\ZZ^n$s are quasi-isometric iff there is a quasi-isometry that respects the tree of spaces and is linear along vertex spaces (see \cite{Whyte10}).   This implies the conjugacy class of  image of the holonomy up to bounded Hausdorff distance is a quasi-isometry invariant.    Then our main result here actually says that the holonomy image is nearly a complete invariant for the most graphs of $\ZZ^n$s.   To make precise the exact conditions under which Hausdorff close images (up to conjugacy) is sufficient to build a map bewteen trees as above we need to examine the possible behavior of holonomy maps.   In particular, if $e \in T$ is an edge, dividing $T$ into halfspaces $T^+$ and $T^-$, we want to understand how $h(T^+)$ and $h(T^-)$ compare to $h(T)$.   If $\overline{h(T^+)} = \overline {h(T)}$ then we say that $T^+$ is {\bf full}.\\

We have the following trichotomy: 

\begin{theorem} \label{trichotomy} Every graph of $\ZZ^n$s satisfies exactly one of:
\begin{itemize}
\item (Proper holonomy) $h: T \to \GL_n(\RR)$ is a proper map.
\item (Parabolic holonomy)   For every $e$ in $T$ exactly one side is full (and then $\Gamma$ is an ascending HNN extension).
\item (Folded holonomy)  For every $e$ in $T$ both sides are full.
\end{itemize}
\end{theorem}

\begin{proof}
This follows from a sequence of observations about the possible behavior of the holonomy map, largely using the fact that the $\Gamma$ action on $T$ is cocompact and $h$ is equivariant with respect to the $hol$ homomorphism. In what follows we will assume $T$ is bushy as it is needed later, but the trichotomy holds in general.

{\bf Claim}. If there is an edge $e_0$ such that neither $T^-$ nor $T^+$ is full then $h$ is proper.

\begin{proof} Assume by contradiction that $h$ is not proper. Then there is a sequence $\gamma_1, \gamma_2, \ldots$ in $\Gamma$ so that $hol(\gamma_i) \to {Id}$ in $\GL_n(\RR)$ and $\gamma_i v\to \infty$ in $T$ for any $v \in T$.   By passing to a further subsequence we may assume $\gamma_i v \to p \in \partial T$.    Let $T^+$ be the half-space of $e_0$ that contains $p$.

For any $a \in hol(\Gamma)$ choose a $u$ in $T$ with $h(u) = a$.    By construction we have $h(\gamma_i u) = hol(\gamma_i)h(u) \to h(u)=a$, and $\gamma_i u \to p \in \partial T^+$, so $\gamma_i u\in T^+$ for $i\gg 0$.   Thus $a \in \overline{h(T^+)}$.   Thus $T^+$ is full, a contradiction, so $h$ must be proper. 

\end{proof}

\begin{lemma} If there is an edge $e_0$ such both  $T^-$ and $T^+$ are full then the same is true for every edge.  
\end{lemma}

\begin{proof}

 Let $e$ be an edge of $T$ and $S$ one of its sides.    By co-compactness there is a $\gamma \in \Gamma$ so that $\gamma e_0 \in S$.    Thus $S$ contains $\gamma T^-$  or $\gamma T^+$, and since each of these is full and $h(\gamma T^{\pm}) = hol(\gamma) h(T^{\pm})$, we have that $S$ is also full.

\end{proof}

Finally, if neither of the previous claims applies then we must have every edge having exactly one side full.  If we orient edges to point towards the full side then we get a $\Gamma$ invariant orientation on $T$, and each vertex must have exactly one incident edge oriented away from it.    This descends to the graph $T/{\Gamma}$, which is finite.     Thus each vertex in $T/ {\Gamma}$ must have exactly one edge pointed in and one pointed out.  This implies $T/{\Gamma}$ is a circle.    Further, since the edge pointed out of a vertex has only one lift at each vertex of $T$ above it, the inclusion of the edge group to the vertex group for this edge is an isomorphism.    Thus $\Gamma$ is an ascending HNN extension of $\ZZ^n$. 

\end{proof}

\begin{remark} It is not hard to see that proper holonomy means $\Gamma$ is virtually a semi-direct product $\ZZ^n \rtimes F$ for $F$ a discrete free subgroup of $\GL_n(\ZZ)$, and we have seen that parabolic holonomy means $\Gamma$ is virtually an ascending HNN extension of $\ZZ^n$.    Thus the folded case is the generic situation, in particular it includes the groups constructed by Carette.    We show below that for those groups the Hausdorff class of the holonomy is actually a complete QI invariant.   The same is immediate for the proper case.   For the parabolic case the situation is significantly more subtle and rigid, see \cite{FarbMosher00}.  The full quasi-isometry classification of graphs of $\ZZ^n$s is given in \cite{Whyte10}.  The case of $n=1$ is older, there the holonomy image is, up to bounded distance, either trivial or everything (see \cite{FMbs}, \cite{WhyBS}).
\end{remark}

\subsection{Folded holonomy and lifting}

We now return to the proof of the main result, Theorem \ref{FoldedQI}.   Let $\Gamma$ and $\Gamma'$ be bushy graphs of $\ZZ^n$s with folded and coarsely equivalent holonomy.  By lemma  \ref{QIfromTreeMap}, we can show $\Gamma$ and $\Gamma'$ are quasi-isometric if we find a quasi-isometry $T \to T'$ that, after conjugating, commutes up to bounded distance with holonomy maps.      We can conjugate the holonomy map by changing  the model space for $\Gamma$, so we will assume we have done that first and have the images of the holonomy at finite Hausdorff distance and are looking to build a map $T \to T'$ commuting with the holonomy within bounded distance. 

We first need to refine what we know about holonomy maps in the folded case, making the fullness of the half spaces quantitative.   We call {\bf directed path lifting} this quantitative statement as in lemma \ref{directedlift}.

\begin{lemma} \label{directedlift}Let $\Gamma$ be a graph of $\ZZ^n$s with folded holonomy.    For any $R_0 > 0$ and $R>0$ there is $D>0$ so that for any $e = (v^-,v^+)$ in $T$  an edge and $T_+$ a side, and any $a \in h(T)$ with $d(a, h(v^-)) \leq R$ then there is $v \in T_+$ with $d(v,v^+) \leq D$ and $d(h(v),a)\leq R_0$.
\end{lemma}

\begin{proof}

Suppose not.  Then we have an $R_0$ and $R$ for which there is no such $D$.    Let $e_i \in ET$ and $a_i \in h(T)$ be a sequence with $d(h({v_i}^-), a_i) \leq R$ so that for all $u$ in ${T^+}_i$ within $i$ of ${v_i}^+$ we have $d(h(u),a_i) > R_0$.

Since the $\Gamma$ action is cocompact we can pass to a subsequence and translate to reduce to the case that $e_i = e $ is the same throughout the sequence.   Further, since $\GL_n(\RR)$ is locally compact, we can pass to a further subsequence such that $a_i \to a \in \overline{h(T)}$.      Thus there is some $i_0$ so that $d(a_i,a) \leq \frac{R_0}{2}$ for $i > i_0$.   Thus we have that for every $i > i_0$ there is no $u \in T^+$ within $i$ of $v^+$ so that $d(h(u),a) \leq \frac{R_0}{2}$ since such a $u$ would have $d(h(u),a_i) \leq \frac{R_0}{2} + \frac{R_0}{2} = R_0$.    As $i \to \infty$  this says that $h(T^+)$ is disjoint from the ball of radius $\frac{R_0}{2}$ around $a$, so $a \notin \overline{h(T_+)}$, contradicting that $T^+$ is full.

\end{proof}
 
\subsection{Building the quasi-isometry}

In view of lemma \ref{QIfromTreeMap}, the next result implies Theorem \ref{FoldedQI}.
 
\begin{theorem}\label{TreeMapsExist}  Fix $H \subset \GL_n(\RR)$.   All two bushy graphs of $\ZZ^n$s with folded holonomy coarsely equivalent to $H$ have Bass-Serre trees that are quasi-isometric over $\GL_n(\RR)$.   
\end{theorem}

\begin{proof}

 The first step is to show that  directed path lifting allows us to lift maps from $\GL_n(\RR)$ to $T$:
 
 \begin{lemma}\label{lift} Let $\Gamma$ be a bushy graph of $\ZZ^n$s with directed path lifting, and let $T_0$ be a tree of bounded valence with Lipschitz map $h_0: T_0 \to \GL_n(\RR)$ with image in some $D$-neighborhood of $hol(\Gamma)$.    There are $(K,C)$ so that for any halfspace $T^+$ of $T$ and vertices $u_0 \in T^+$ and $v_0 \in T_0$ so that $d(h(u_0),h_0(v_0))\leq R_0$ there is a subtree $T''$ of $T^+$ containing $u_0$ and a quasi-isometry $\phi: (T_0,v_0)  \to (T'',u_0)$ so that $d(h(\phi(v)),h_0(v)) \leq C$ for all $v \in T_0$.   Further, the vertices of $T''$ which are of lower valence in $T''$ than in $T$ are $K$-coarsely dense. 
 \end{lemma}
 
 \begin{proof}
 
First note that we can change $h_0$ a bounded distance so that $h_0(T_0) \subset h(T)$ and $h_0(v_0) = h(u_0)$.   Further, we can shrink $T^+$ to a smaller half-space, so wlog we can assume $u_0$ is an edge $e_0$ with $T^+$ one side of $e_0$.   Take $r_{0}$  large enough so that any hemisphere (sphere around an edge intersected with one halfspace of that edge) of radius $r_{0}$ in $T$ has strictly larger cardinality than the valence of $T_0$.  This is possible since $T$ is bushy and cocompact and $T_0$ has bounded valence.  

We will build the lift inductively on the ball of radius $n$ in $T_0$  around $v_0$ with the further property at each stage  that the map is injective on vertices and takes the ball of radius $n$ into  a subtree $S_n$ of $T$ whose extreme points include the image of the points in the sphere of radius $n$.   For $n=0$ we simply map $v_0$ to $u_0$, so $S_0 = \{u_0\}$ and $\phi_0 (v_0) = u_0$. 

We next explain how to extend from the ball of radius $n$ to $n+1$.  Let $v$ be a vertex in the sphere of radius $n$ in $T_0$.    Since $u=\phi_n(v)$ is an extreme point of $S_n$, if $n>0$ there is a unique edge $e$ of $S_n$ at $u$.   Let $H$ be the halfspace of $e$ in  $T$ that contains $u$ (note that $H$ is a subset of $T^+$).     If $n=0$ we take $H = T^+$.  Let $w_1,w_2, \ldots, w_k$ be the vertices in $T_0$ adjacent to $v$ at distance $n+1$.    By the definition of $r_0$ we can find distinct vertices $u_1,\ldots, u_k$ of $H$ at distance $r_0$ from $u$.   For each $u_i$ choose an edge $e_i$ at $u_i$ connecting to a vertex at distance $r_0 + 1$ and let $H^+_i$ be the halfspace of $e_i$ disjoint from $S_n$.   By directed path lifting for $T$ there is some uniform $R$ so that we can find $u'_1, u'_2,\ldots, u'_k$ with $u'_i \in H^+_i$ within $R$ of $u_i$ and with $d(h(u'_i),h_0(w_i)) \leq C$.   We then extend $S_n$ by adding the geodesics from $u$ to $u'_i$ and $\phi_{n+1}$ mapping the edge $\overline{v w_i}$ to this geodesic.     Since the lengths of these geodesics are bounded by $K = r_0 + 1 + R$, the extension is $K$-Lipschitz,  the images of distinct vertices are at least $r_0$ apart and $K$-coarsely dense in $S_{n+1}$.   Thus the extension has the desired properties.      By the definition of $r_0$ the valence of $v$ in $S_{n+1}$ is strictly smaller than its valence in $T$, so the vertices of $T''$ which are of lower valence in $T''$ than in $T$ are $K$-coarsely dense as claimed.

\end{proof}

We return to the proof of  Theorem \ref{TreeMapsExist} (using  lemma \ref{lift} as a key tool).     If $T$ and $T'$ are two trees with folded holonomy coarsely equivalent to some $H \subset \GL_n(\RR)$ then the lemma allows us to build (many) quasi-isometric embeddings of $T$ into $T'$ over $\GL_n(\RR)$ or vice versa.    These maps are almost certainly not onto, so we need a less direct construction. Fix any one bushy Bass-Serre tree $T_0$ with $h_0 :T_0\to \GL_n(\RR)$ having image Hausdorff equivalent to $H$ and  we will then show that any $T$ with folded holonomy equivalent to $H$ is quasi-isometric over $\GL_n(\RR)$ to a fixed tree $\hat{T_0}$ built from $T_0$.   The proof is somewhat technical so we begin by describing the outline:

\begin{itemize}
\item Use the lemma to show that $T$ has a spanning forest $F$ with every component $C$ a subtree that is quasi-isometric to $T_0$ over $\GL_n(\RR)$, and such that the extremities in $T$ of the edges connecting $C$ to  $T \setminus C$, are coarsely dense in $C$.
\item Show that $T$ is quasi-isometric over $\GL_n(\RR)$ to a tree $T'$ built by replacing each component with a copy of $T_0$ in which every vertex of $T'$ is adjacent to exactly one edge not in one of the copies of $T_0$.
\item Show that this resulting $T'$ is unique, call it $\hat{T_0}$, so all such $T$ are quasi-isometric over $\GL_n(\RR)$ to a fixed $\hat{T_0}$, proving Theorem \ref{TreeMapsExist}.
\end{itemize}

For the first step: We show that there is a vertex covering $F$ of $T$ by disjoint subtrees which are lifts of $T_0$ as in lemma \ref{lift} so that the edges not in $F$ attach at a coarsely dense set of vertices of each component $C$ and such that the quasi-isometries from each subtree $C$ to $T_0$ assemble to a Lipschitz map $T \to T_0$.  We build this inductively, covering larger and larger subtrees of $T$.  Start the induction by finding a single subtree $T''$ from the lemma that is a lift of $T_0$ through a basepoint of $T$.   This gives a cover of an initial subtree $S$ of $T$, where $S=T''$ is covered by a single lift of $T_0$.  

Given any subtree $S$ of $T$ covered by a forest $F$ of lifts of $T_0$, if $ S \neq T$ then  there is an edge $e$ with one endpoint $e_-$ in $S$, and one endpoint $e_+$ not in $S$.  
Denote by $T^+$ the halfspace of $T$ defined by $e$ and containing $e_+$, so that $T^+$ is disjoint from $S$. The vertex $e_-$ belongs to some component of the forest $F$, so it is mapped to some vertex $v\in T_0$ under the quasi-isometry between that component and $T_0$. Since the distance between $e_-$ and $e_+$ is 1, the image $h(e_+)$ is close to $h(e_-)$, which by construction is close to $h_0(v)$. By the bushy condition there is $v_0\in T_0$ with $h_0(v_0)$ within a fixed $R_0$ from $h(e_+)$. By lemma \ref{lift} (applied with $u_0=e_+$ and the above $v_0$) we can build a subtree $T''_e$ of $T^+$, with $e_+\in T''_e$, and $(T''_e,e_+)$ quasi-isometric to $(T_0,v_0)$. This expands $S$ to $S \cup \{e\} \cup T''_e$.   Thus a maximal subtree with such a cover by a forest of lifts of $T_0$ must be all of $T$.   
Note that, by construction, there is an $R$ so that for each edge $e$ of $T$ connecting distinct components $C$ and $C'$, under the maps from $C$ and $C'$ to $T_0$ the endpoints of $e$ are within  $R$ of each other.    Further, for any component $C$ of $F$, the edges in $T$ connecting $C$ to $T \setminus C$ connect to $C$ at a $K$-coarsely dense set of vertices of $C$.

For the second step we construct $T'$ in two stages.    First we modify $T$ by changing each component of $F$ to be a copy of $T_0$ as follows:   Let $Q$ be the (countable valence) tree obtained from $T$ by collapsing each component of $F$ to a point.    Build $T'_1$ starting with the forest $VQ \times T_0$.   For each edge $e$ of $Q$ there is a corresponding edge $\bar{e}$ of $T$ connecting two components of $F$.   Each of these components comes with a quasi-isometry to $T_0$ (over $\GL_n(\RR)$ ) since they are built as lifts of $T_0$.   We add an edge to $T'_1$ connecting the corresponding copies of $T_0$ at the images of the endpoints of $\bar{e}$ under these quasi-isometries.    The resulting tree $T'_1$ has a spanning forest $F_1$ and a coarse Lipschitz map $T'_1 \to T_0$ which is an isomorphism on each component of $F_1$.

We produce $T'$ from $T'_1$ by changing where the edges outside of $F_1$ are attached, moving them a bounded distance on each end so that two things hold:  first, each edge $e$ attaches to the same vertex of $T_0$, call it $l(e)$, in the components of $F_1$ on each end.  Second, every vertex of $T'$ is an endpoint of exactly one such edge.      We can reformulate this in terms of the quotient tree $Q$: we need to find a labelling $l$ of the edges of $Q$ so that for each $e$, $l(e)$ is within some uniform $R$ of the endpoints $\bar{e}$ attaches to in $T_0$ on each side, and such that for each vertex of $Q$ the adjacent edges are labelled bijectively by $VT_0$.  

The technical lemma we need for this is lemma 4.5 from \cite{thesis}  : Given $m$ and $K$ there is $D$ so that if $S$ is a set mapping to $VT_0$ with $K$-coarsely dense image and of multiplicity bounded by $m$ then there is a bijection of $S$ with $VT_0$ within $D$ from the original mapping, and that this bijection can be chosen to map a basepoint $s_0$ to an arbitrary point of $VT_0$ within $D$ of the original mapping.   Using this,we produce the labelling $l$  of the edges of $Q$ by starting at a basepoint $q_0$ of $Q$ and expanding radially.  Suppose we have a labelling of all the edges adjacent to vertices within $n$ of $q_0$ , and let $v$ be a vertex of $Q$ at  distance $n+1$ connecting to $u$ at distance $n$ along an edge $e$.    We then apply the technical lemma to the set of edges at $v$ using $e$ and its given label as our basepoint, the resulting bijection gives the extension of the labelling to all edges at $v$.     By induction we get the needed labelling, completing the second step. 

The third step is almost immediate.  Given $Q$ and the labelling of its edges we can reconstruct $T$ as $VQ \times T_0$ with one edge $\bar{e}$ added for each $e$ in $EQ$,  where if $e $ connects vertices $u$ and $v$ then $\bar{e}$ connects $u \times l(e)$ to $v \times l(e)$.    However, there is a unique such labelled tree: if $Q'$ is another tree with edges labelled by $VT_0$ so that each vertex has exactly one adjacent edge  with each label then we can map any basepoint of $Q$ to a basepoint of $Q'$ and then the labelling on the edges uniquely determines a (label preserving) isomorphism.   The corrsponding trees $T$ and $T'$ are then isomorphic, and in a way that is the identity map between $T_0$ components.   As a model,  one such  labelled tree $Q_0$  is the Cayley graph of the free group on generating set $VT_0$, and we build a tree $\hat{T_0}$ from this as above, mapping $\hat{T_0}$ to $\GL_n(\RR)$ by first mapping $\hat{T_0}$ to $T_0$ followed by the given map  $T_0$ to $\GL_n(\RR)$.    We have shown that  all bushy trees  $T$ with folded holonomy and image $H$  are quasi-isometric over $\GL_n(\RR)$ to $\hat{T_0}$,  completing the proof of Theorem \ref{FoldedQI}.

\end{proof}
	
\appendix	
	\section*{Appendix B: Equivariant compression is not a QI-invariant}
	\begin{center} \textsc{by S. Arnt, T. Pillon and A. Valette} \end{center}

	Equivariant compression was introduced as a way to quantify the Haagerup property. For $G$ a finitely generated group, $p\in[1,+\infty[$, and $f:G\rightarrow L^p$ a $G$-equivariant map (with respect to some affine isometric action of $G$ on $L^p$), the $L^p$-compression of $f$ is: 
	$$\comp_p(f)=\sup \{\alpha\in[0,1]: \exists C>0: C|x^{-1}y|_S^\alpha - C \leq \|f(x)-f(y)\|_p \;\forall x,y\in G\},$$
	where $|.|_S$  denotes word length with respect to the finite generating set $S$ in $G$; clearly $\comp_p(f)$ does not depend on the choice of $S$. 
	Then the $L^p$-compression of $G$ is $\alpha^\sharp_p(G)=\sup \comp_p(f)$, where the supremum is taken on all $G$-equivariant maps $f:G\rightarrow L^p$. It is a result by Naor and Peres (see \cite[Thm. 9.1]{NP}) that $\alpha_p^\sharp$ is a QI-invariant among finitely generated amenable groups. The purpose of this Appendix is to use Carette's examples from Theorem 1 above, to show that, as it might be expected, $\alpha_p^\sharp$ is not QI-invariant among all finitely generated groups. 
	
	\begin{theorem} For $1\leq p\leq 2$, the vanishing of $\alpha_p^\sharp$ is not a QI-invariant.
	\end{theorem}
	
	\begin{proof} We use the two groups $\Gamma$ and $\Lambda$ from the proof of Theorem \ref{theorem:main}. We will show that, for $1\leq p\leq 2$, we have $\alpha_p^\sharp(\Gamma)=\frac{1}{p}$ while 
	$\alpha^\sharp_p(\Lambda)=0$.
	\begin{enumerate}
	\item[1)] $\alpha_p^\sharp(\Lambda)=0$ for $p\in[1,2]$. Indeed, assume by contradiction $\alpha^\sharp_p(\Lambda)>0$ for some $p$. Then $\Lambda$ admits a proper affine isometric action on $L^p$. By Corollary 6.23 of \cite{CDH} (using $1\leq p\leq 2$), the group $\Lambda$ has the Haagerup property, which is a contradiction.
	\item[2)] For every $p\geq 1$, we have $\alpha_p^\sharp(\Gamma)=\max\{\frac{1}{p},\frac{1}{2}\}$. This will follow from Theorem 7.3 in \cite{CornulierValette12}, of which we recall the relevant part.
	\begin{theorem} For $G\in \GBS_n$, non-amenable, assume the following 3 conditions:
	\begin{itemize}
	\item $\overline{hol(G)}$ is amenable;
	\item there exists a closed, almost connected subgroup $L\subset \GL_n(\RR)$ in which $\overline{hol(G)}$ is co-compact;
	\item the inclusion $\RR^n\rtimes \hol(G)\rightarrow \RR^n\rtimes L$ induces a quasi-isometry in restriction to $\RR^n$, where $\hol(G)$ is endowed with the discrete topology in the first semi-direct product, and both semi-direct products are endowed with the word length associated to a compact generating subset. 
	\end{itemize}
	Then $\alpha^\sharp_p(G)=\max\{1/p,1/2\}$ for $p\geq 1$. $\square$
	\end{theorem}

	This result is proved by constructing a homomorphism $G\rightarrow Aut(T)\times F\times (\RR^n\rtimes L)$ which is a quasi-isometric embedding (where $T$ is the Bass-Serre tree of $G$ and $F$ is some free group) and estimating $\alpha^\sharp_p(Aut(T)\times F\times (\RR^n\rtimes L))$.
	
	We now check that the group $\Gamma$ of Theorem \ref{theorem:main} satisfies the three assumptions of the previous result. First, $\overline{\hol(\Gamma)}$ is clearly amenable. Second, $\overline{\hol(\Gamma)}$ is co-compact in the upper triangular subgroup $L$ of $\SL_2(\RR)$, which is almost connected. Finally a sufficient condition for the third condition to hold, appears in Proposition 2.7(b) of \cite{CornulierValette12}: it is enough that $\RR^2$ is exponentially distorted in $\RR^2\rtimes \hol(\Gamma)$, i.e. $\limsup_{m\rightarrow\infty}\frac{|mv|_S}{\log m}<\infty$ for every vector $v\in\RR^2$ (where $|.|_S$ is word length with respect to some compact generating set $S$). This holds here because the matrix $H=\left(\begin{array}{cc}2 & 0 \\0 & \frac{1}{2}\end{array}\right)$ lies in $\hol(\Gamma)$.\qedhere
	\end{enumerate}
	\end{proof}
	
	We thank Mathieu Carette for offering us to write up our result as an Appendix to his paper.


\end{document}